\documentclass[11pt,reqno]{amsart}
\usepackage[T1]{fontenc}
\usepackage[utf8]{inputenc}
\usepackage{lmodern}
\usepackage[a4paper,textwidth=6.1in,textheight=8.8in,centering]{geometry}
\usepackage{amsmath,amssymb,mathtools,mathrsfs}
\usepackage{microtype}
\usepackage{enumitem}
\usepackage{xcolor}
\usepackage[unicode,colorlinks=true,linkcolor=blue!45!black,
  citecolor=blue!45!black,urlcolor=blue!45!black,breaklinks=true]{hyperref}
\hypersetup{pdftitle={Meromorphic solutions of first-order differential equations with rational exponential coefficients},
  pdfauthor={Deguang Zhong},pdfsubject={Private research draft on Gundersen Question 6.2},
  pdfkeywords={Malmquist theorem, meromorphic solution, exponential coefficient, ramification, rational surface}}
\allowdisplaybreaks[1]
\numberwithin{equation}{section}
\setlist[enumerate]{label=\textup{(\roman*)},leftmargin=2.2em,itemsep=3pt}
\newtheorem{theorem}{Theorem}[section]
\newtheorem{proposition}[theorem]{Proposition}
\newtheorem{lemma}[theorem]{Lemma}
\newtheorem{corollary}[theorem]{Corollary}
\theoremstyle{definition}

\newtheorem{example}[theorem]{Example}
\theoremstyle{remark}
\newtheorem{remark}[theorem]{Remark}
\newcommand{\C}{\mathbb C}
\newcommand{\Z}{\mathbb Z}
\newcommand{\N}{\mathbb N}
\newcommand{\Q}{\mathbb Q}
\newcommand{\PP}{\mathbb P}
\newcommand{\OO}{\mathcal O}
\newcommand{\MM}{\mathcal M}
\newcommand{\Erad}{\mathscr E_{\mathrm{rad}}}
\newcommand{\ord}{\operatorname{ord}}

\newcommand{\dd}{\,\mathrm d}
\newcommand{\doi}[1]{\href{https://doi.org/#1}{\nolinkurl{doi:#1}}}
\makeatletter
\renewcommand{\@settitle}{\begin{center}\baselineskip=18pt\relax
  {\Large\bfseries\@title\par}\end{center}}
\makeatother

\title[Meromorphic solutions with exponential coefficients]
{Meromorphic solutions of first-order differential equations\\
with rational exponential coefficients}

\author{Zhong Deguang}
\address{Institute of Applied Mathematics, Shenzhen Polytechnic University,
Shenzhen 518055, China}
\email{huachengzhon@163.com}

\author{Meng Fanning\textsuperscript{*}}
\address{School of Mathematics and Information Science, Guangzhou University,
Guangzhou, Guangdong 510006, P.~R. China}
\email{mfngzhu@gzhu.edu.cn}

\author{Yuan Wenjun}
\address{School of Basic and General Education,
Guangzhou University of Software, Guangzhou 510990, P.~R. China}
\email{wjyuan1957@126.com}

\thanks{\textsuperscript{*}Corresponding author: Meng Fanning.}

\date{9 September 2026}

\begin{document}

\begin{abstract}
We study first-order differential equations $f'=R(e^z,f)$, where
$R\in\C(t,w)$. We prove that a meromorphic solution on the whole complex
plane is algebraic over $\C(e^z)$ unless $R$ is a polynomial of degree at most
two in its second variable. Such an algebraic solution necessarily has the
form $S(e^{z/q})$, with $S$ rational and $q$ a positive integer, giving an
affirmative answer to Question~6.2 in Gundersen's collection. The geometric
input is Guillot's theorem on single-valued trajectories of meromorphic
vector fields on surfaces. The additional argument compares the fibration
provided by that theorem with the original exponential coordinate. A
ramification calculation forces every finite nonzero branch value $a$ to
satisfy $Da=a$, which excludes such a value and reduces the comparison to
a two-point cover. Divisor divisibility and relative algebraic closedness
then recover a Riccati equation over the original coefficient field. No
growth hypothesis is imposed. We also identify the algebraic degree with
the least integer deck period and describe the growth and value fibres of
the resulting solutions.
\end{abstract}

\maketitle

\section{Introduction}
\label{sec:introduction}

The global single-valuedness of a solution of a complex differential
equation imposes restrictions that are not visible from local existence
theory. The classical theorem of Malmquist~\cite{Malmquist1913} is a
fundamental example. If $R(z,w)$ is rational in both variables and
$f'=R(z,f)$ has a transcendental meromorphic solution in $\C$, then $R$ is
a polynomial of degree at most two in $w$. Thus the equation is Riccati,
with the linear and constant-in-$w$ cases included. Differential-field
formulations and further developments are discussed by
Eremenko~\cite{Eremenko1982} and Laine~\cite{Laine1993}.

Replacing rational coefficients in $z$ by rational coefficients in $e^z$
changes the natural conclusion. The relevant algebraic functions need not
be rational in $e^z$ itself. For example,
\begin{equation}
 f(z)=e^{-z/2},\qquad f'=-\frac12 e^z f^3.
 \label{eq:historical-example}
\end{equation}
The half-exponential is unavoidable in this example. It suggests working
with finite root extensions of the exponential function field.

Gundersen~\cite[Question~6.2]{Gundersen2017} asks whether every global
meromorphic solution of
\begin{equation}
 f'(z)=R(e^z,f(z)),\qquad R\in\C(t,w),
 \label{eq:main-equation}
\end{equation}
has the form $S(e^{cz})$, with $S$ rational and $c\in\Q$, provided that the
equation is neither Riccati nor linear. In his formulation, the question
is attributed to Gundersen and Laine~\cite{GundersenLaine1985}; their
paper appeared in 1985. Their Theorem~2 treats a structured polynomial
subclass and shows that its meromorphic solutions are Laurent polynomials
in $e^{z/m}$ for a positive integer $m$~\cite[Theorem~2]{GundersenLaine1985};
Gundersen's collection also points to Rieppo's Theorem~4.8 for another partial
result~\cite{Rieppo1998}. The 2023 survey of Heittokangas, Ishizaki,
Tohge and Wen both summarize the earlier theorem and restate, as an
open problem, Question~6.2 in the rational-coefficient setting
\cite[Section~4.4]{HeittokangasEtAl2023}.
The present argument compares two fibrations of the same rational
surface, retaining the original exponential coordinate throughout.

Throughout the paper, ``meromorphic'' means meromorphic on the whole
complex plane unless a different domain is specified. A rational
expression is always interpreted after cancellation of common factors.
We say that $R\in\C(t,w)$ is of \emph{Riccati type} if
\begin{equation}
 R(t,w)=A_2(t)w^2+A_1(t)w+A_0(t),
 \qquad A_j\in\C(t).
 \label{eq:riccati-type}
\end{equation}
This terminology includes $A_2=0$ and the equations independent of $w$.
Let
\begin{equation}
 \Erad=\bigcup_{q\geq1}\C(e^{z/q})\subset\MM(\C).
 \label{eq:radical-field}
\end{equation}
The union is a field: any two of its members belong to the field indexed
by a common multiple of their denominators.

\begin{theorem}
\label{thm:main}
Suppose that $R\in\C(t,w)$ is not of Riccati type. Every meromorphic
solution of \eqref{eq:main-equation} belongs to $\Erad$. More explicitly,
for each such solution there are $q\in\N$ and $S\in\C(u)$ such that
\begin{equation}
 f(z)=S(e^{z/q}).
 \label{eq:main-conclusion}
\end{equation}
Equivalently, if a meromorphic solution of \eqref{eq:main-equation} is
transcendental over $\C(e^z)$, then $R$ has the form
\eqref{eq:riccati-type}.
\end{theorem}

No finite-order assumption occurs in this statement. Nor is periodicity
or any restriction on three value fibres assumed. Those properties are
consequences of the conclusion, rather than inputs to the proof.

There are two distinct assertions in Theorem~\ref{thm:main}. The first
is algebraicity over the original coefficient field $\C(e^z)$. The
second is the classification of the single-valued functions that are
algebraic over that field. The latter admits a useful exact version.

\begin{theorem}
\label{thm:primitive-intro}
Let $f\in\MM(\C)$ be algebraic over $\C(e^z)$, and set
\[
 m=[\C(e^z,f):\C(e^z)].
\]
Then $m$ is the least positive integer for which
\begin{equation}
 f(z+2\pi i m)=f(z).
 \label{eq:primitive-period}
\end{equation}
Moreover,
\begin{equation}
 \C(e^z,f)=\C(e^{z/m}).
 \label{eq:primitive-field}
\end{equation}
In particular, $f=S(e^{z/m})$ for a rational function $S$.
\end{theorem}

The period in this theorem is measured inside the prescribed deck group
$2\pi i\Z$. It need not be a primitive period among all complex periods
of $f$. For instance, $e^{2z}$ has $m=1$ in the theorem, although its
least positive imaginary period is $\pi i$.

For a nonconstant solution, the usual Nevanlinna characteristic satisfies
\begin{equation}
 T(r,f)=\frac{\deg S}{\pi m}\,r+O(1)
 \label{eq:growth-intro}
\end{equation}
in a primitive representation. Entire solutions are Laurent polynomials
in $e^{z/m}$. We establish these consequences directly in
Section~\ref{sec:growth}.
Section~\ref{sec:deficiencies} gives the counting functions and
deficiencies explicitly. Only $S(0)$ and $S(\infty)$ can be deficient
values, and the sum of the truncated deficiencies is exactly two.
These formulas express the value distribution of the solution in
terms of the two punctures and the ramification of the rational map $S$.

The proof of algebraicity uses the rational-surface consequence of
Guillot's classification~\cite[Main Theorem]{Guillot2019}: a
Zariski-dense single-valued
trajectory yields, after a birational change, a preserved fibration with
vertical poles. The challenge is to compare its base coordinate with
the prescribed rational function $t$, for which $Dt=t$.

Write $h$ for the new base coordinate. If $h$ and $t$ are independent,
$t$ is a finite map on the geometric generic fibre over $\C(h)$. At a
ramification point above a finite value $a$, preservation of the local
ring gives $Da=a$. A nonzero such value would restrict along the given
solution to a constant multiple of $e^z$, contradicting the independence
of $h$ and $t$. Riemann--Hurwitz then gives
\[
 t=c(h)u^d.
\]
The meromorphic single-valuedness of $u$ along the trajectory controls
the divisor of $c$. Relative algebraic closedness of $\C(t)$ in
$\C(t,w)$ completes the comparison and forces Riccati type in the
\emph{original} variable $w$.

We include the algebraic details of this comparison. In particular,
regularity after extension of the generic coefficient field, descent of
the two distinguished points, and the integer-exponent change of
variables will all be proved explicitly. The paper concludes with
examples and a comparison with the additive coordinate $Dt=1$.

\section{Differential function fields and rational curves}
\label{sec:fields}

All fields in this paper have characteristic zero. A derivation is
$\C$-linear and satisfies the Leibniz rule. If $K\subset L$ are
differential fields, the restriction of the derivation to $K$ is
understood. In particular, the derivation on a curve over $K$ need not
be $K$-linear.

\subsection{Relative algebraic closedness}

We say that $K$ is \emph{relatively algebraically closed} in $L$ if every
element of $L$ algebraic over $K$ already belongs to $K$. We will use
this property for two different subfields, and it is important to keep
their roles separate.

\begin{lemma}
\label{lem:relative}
If $x$ is transcendental over a field $F$, then $F$ is relatively
algebraically closed in $F(x)$. Consequently, if $v\in\C(t,w)$ and
\[
 v^d=\gamma t,\qquad \gamma\in\C^*,\quad d\in\N,
\]
then $d=1$.
\end{lemma}

\begin{proof}
If $v\in F(x)\setminus F$, write $v=P(x)/Q(x)$ with $P,Q\in F[X]$.
The relation $P(x)-vQ(x)=0$ shows that $x$ is algebraic over $F(v)$.
If $v$ were algebraic over $F$, then $x$ would be algebraic over $F$,
which is impossible. Thus no such $v$ is algebraic over $F$.

For the second assertion, take $F=\C(t)$ and $x=w$. The proposed $v$
is algebraic over $F$, and therefore lies in $\C(t)$. Taking orders at
$t=0$ gives
\[
 d\,\ord_0 v=1.
\]
Since the order of a rational function is an integer, $d=1$.
\end{proof}

When $S\to B$ is a fibration of smooth projective varieties with
connected fibres and $B$ is a smooth curve, Stein factorization gives
relative algebraic closedness of $\C(B)$ in $\C(S)$. Conversely, the
finite part of Stein factorization records the relative algebraic
closure of the base function field. We use this standard
function-field interpretation of connectedness; see
\cite[Chapter~III, Section~11]{Hartshorne1977}.

\subsection{Extension of the derivation and local regularity}

If $\alpha$ is algebraic over a differential field $K$, its derivative
is uniquely determined. Indeed, for a separable minimal polynomial
$p(Y)=\sum_j a_jY^j$, differentiating $p(\alpha)=0$ gives
\begin{equation}
 D\alpha=-\frac{\sum_j(Da_j)\alpha^j}{p'(\alpha)}.
 \label{eq:algebraic-derivative}
\end{equation}
The denominator is nonzero. This construction extends $D$ uniquely to
every algebraic extension of $K$ and hence to an algebraic closure
$\overline K$; see, for example, Kolchin~\cite{Kolchin1973} for the
differential-algebraic background.

Let $\Gamma/K$ be a smooth projective geometrically integral curve,
with function field $L$. A derivation $D:L\to L$ extending one on $K$
will be called \emph{regular on the geometric curve} if its extension
to $\overline K(\Gamma)$ satisfies
\begin{equation}
 D\OO_{\Gamma_{\overline K},p}
 \subseteq\OO_{\Gamma_{\overline K},p}
 \qquad\text{for every }p\in\Gamma_{\overline K}.
 \label{eq:geometric-regularity}
\end{equation}
This is the local ring condition appearing in the differential-field
formulation of the Fuchs conditions; compare
\cite[Section~1, Definition~1]{Eremenko1982}. It does not assert that
$D$ preserves the maximal ideal of the local ring. The distinction will
matter in the ramification calculation.

\begin{lemma}
\label{lem:base-change}
Let $\pi:S\to B$ be a morphism from a smooth projective complex surface
to a smooth projective curve, with connected fibres. Suppose that a
meromorphic vector field $X$ preserves the fibration and that its polar
divisor is contained in finitely many fibres. Put
$K=\C(B)$ and $L=\C(S)$. Then the corresponding derivation on $L$
preserves $K$ and is regular on the geometric generic fibre.
\end{lemma}

\begin{proof}
Preservation of the fibration implies $D(K)\subseteq K$. Remove from
$B$ the images of the polar components and the finitely many critical
values of $\pi$. On the remaining open set, the fibration is smooth and
$X$ is holomorphic on the whole inverse image. Every affine coordinate
ring on this inverse image is therefore preserved by $D$.

Passing to the generic fibre consists of localizing by nonzero functions
of the base. Derivations preserve these localized rings, because
\[
 D(a/s)=\frac{sDa-aDs}{s^2}.
\]
For an affine coordinate ring $A$ of the generic fibre, extend the
derivation to $A\otimes_K\overline K$ by
\begin{equation}
 D(a\otimes\alpha)=Da\otimes\alpha+a\otimes D\alpha.
 \label{eq:tensor-derivation}
\end{equation}
This formula is balanced over $K$ and thus well-defined. The curve is
smooth and geometrically integral, so these rings are the affine rings
of its geometric base change. Localizing once more at any point proves
\eqref{eq:geometric-regularity}.
\end{proof}

\begin{remark}
\label{rem:semilinear}
Poles of the coefficients $D\alpha$ as algebraic functions on a finite
cover of the base do not contradict the lemma. In a local ring of the
\emph{generic fibre}, every nonzero coefficient in $\overline K$ is a
unit. The assertion concerns poles in the fibre variable, not the
locations of poles on the base curve.
\end{remark}

\subsection{The Riccati criterion}

\begin{lemma}[Regular derivations on a rational curve]
\label{lem:riccati}
Let $K$ be a differential field and let $x$ be transcendental over $K$.
For an extension of the derivation to $K(x)$, the following conditions
are equivalent:
\begin{enumerate}
\item $D$ is regular on $\PP^1_{\overline K}$;
\item $Dx=a_2x^2+a_1x+a_0$, where $a_0,a_1,a_2\in K$.
\end{enumerate}
\end{lemma}

\begin{proof}
Write $Dx=Q(x)$ with $Q\in K(x)$. If $Q$ has a finite pole
$\alpha\in\overline K$, the element $x-\alpha$ is regular at that point
but
\[
 D(x-\alpha)=Q(x)-D\alpha
\]
has a pole. Thus (i) implies that $Q$ is a polynomial. At infinity use
$s=1/x$. Regularity of
\[
 Ds=-s^2Q(1/s)
\]
then implies $\deg Q\leq2$.

Conversely, a polynomial $Q$ of degree at most two makes both $Dx$ on
the affine chart and $D(1/x)$ on the chart at infinity regular. The
derivation of coefficients remains in $\overline K$, and localization
preserves the derivation. Hence every local ring is preserved.
\end{proof}

The criterion is compatible with changes of generator. In particular,
it includes the usual invariance of Riccati equations under
coefficient-dependent fractional linear transformations.

\begin{lemma}
\label{lem:rational-image}
Suppose that $D$ is regular on $\PP^1_{\overline K}$ with coordinate
$x$. Let $F\in K(x)\setminus K$. If $DF=Q(F)$ for some $Q\in K(Y)$,
then $Q$ is a polynomial of degree at most two.
\end{lemma}

\begin{proof}
The nonconstant map $F:\PP^1_{\overline K}\to\PP^1_{\overline K}$
is surjective. If $Q$ has a finite pole at $a$, choose a point above
$a$. The derivative of the regular element $F-a$ would have a pole,
which is impossible. Thus $Q$ is polynomial. At a pole of $F$, the
element $1/F$ is regular, and
\[
 D(1/F)=-Q(F)/F^2.
\]
This is regular only when $\deg Q\leq2$.
\end{proof}

\begin{corollary}
\label{cor:change-generator}
Suppose that $K(x)=K(y)$ and $Dy$ is a polynomial of degree at most two
in $y$. Then $Dx$ is a polynomial of degree at most two in $x$.
\end{corollary}

For clarity, the algebraic change of variables in this corollary is
always fractional linear. A nonconstant rational function $F(y)$ has
extension degree $[K(y):K(F)]=\deg F$; consequently $K(F)=K(y)$ forces
$\deg F=1$. The corresponding matrix is allowed to have entries in
$K$, and differentiation of those entries does not alter the conclusion
of the corollary.

\section{Single-valued trajectories and a preserved fibration}
\label{sec:guillot}

\subsection{The trajectory associated with a transcendental solution}

Assume for the moment that a solution $f$ of
\eqref{eq:main-equation} is transcendental over $\C(e^z)$. Set
\begin{equation}
 L=\C(t,w),\qquad
 D=t\partial_t+R(t,w)\partial_w.
 \label{eq:ambient-derivation}
\end{equation}
The substitution
\begin{equation}
 \iota:L\hookrightarrow\MM(\C),\qquad
 \iota(t)=e^z,\quad \iota(w)=f
 \label{eq:evaluation}
\end{equation}
is an injective field homomorphism. It is differential by the equation
for $f$. The nonvanishing of denominators as meromorphic functions
follows from algebraic independence. In particular, every rational
function on any birational model gives a well-defined meromorphic
function after substitution.

The map
\[
 \Phi:\C\longrightarrow S_0=\PP^1\times\PP^1,
 \qquad \Phi(z)=(e^z,f(z)),
\]
is holomorphic as a map to the compact surface. It is a solution of
$X=D$ wherever $X$ is holomorphic, including at poles of $f$ when a
reciprocal coordinate is used. Its image is Zariski dense, because a
polynomial relation on the image would contradict injectivity of
\eqref{eq:evaluation}. By Chow's theorem
\cite[Appendix~B]{Hartshorne1977}, the image is also contained in no
proper closed analytic subset of the projective surface.

\begin{lemma}
\label{lem:maximal}
Let $P$ be the polar divisor of $X$ on $S_0$. Then
$\Omega=\C\setminus\Phi^{-1}(P)$ is the domain of a single-valued
maximal solution of $X$ on $S_0\setminus P$.
\end{lemma}

\begin{proof}
The preimage $\Phi^{-1}(P)$ is a closed discrete set, since the dense
image of $\Phi$ is not contained in $P$. Its complement $\Omega$ is
connected. Moreover, the restricted image $\Phi(\Omega)$ is still
Zariski dense: an algebraic relation vanishing there would, after
composition with $\Phi$, vanish on a nonempty open subset of $\C$ and
hence everywhere by the identity theorem.

If a sequence in $\Omega$ converges in $\C$ but not in $\Omega$, its
limit is a point $z_0\in\Phi^{-1}(P)$. Continuity of the original map
gives $\Phi(z)\to\Phi(z_0)\in P$, so the restricted trajectory leaves
every compact subset of $S_0\setminus P$. This is exactly the
proper-graph condition in \cite[Definition~5]{Guillot2019}; hence the
solution is univalent maximal in Guillot's terminology.
\end{proof}

Here single-valuedness concerns the time parameter; it does not require
$\Phi$ to be injective. Deleting an isolated time at which the trajectory
meets $P$ is legitimate, because the maximality is asserted in
$S_0\setminus P$, rather than in the original compact surface.

\subsection{The rational-surface consequence of the classification}

We use the following consequence of Guillot's main theorem and the
observation immediately following it.

\begin{theorem}
\label{thm:guillot}
A meromorphic vector field on a projective rational surface with a
Zariski-dense single-valued maximal trajectory has a smooth birational
model admitting a connected fibration over $\PP^1$ such that its polar
divisor is vertical and its flow maps fibres to fibres.
\end{theorem}

\begin{proof}
Guillot's Main Theorem~\cite[Main Theorem and the paragraph following
it]{Guillot2019} gives either a preserved fibration or, in the only
nonfibred case, a holomorphic vector field on an Abelian surface without
elliptic subgroups. The latter surface is not birational to a rational
surface: their irregularities are respectively two and zero, and
irregularity is a birational invariant; see
Beauville~\cite{Beauville1996}.

The base of the preserved fibration must be rational. Indeed, a dominant
map from a rational surface to a positive-genus curve would pull back a
nonzero holomorphic one-form, whereas a rational surface has none.
On the smooth model supplied by the theorem, Stein factorization gives
connected fibres over $\PP^1$. Guillot's conclusion places every polar
component in a fibre; passing to this connected-fibre factorization does
not rescale the vector field. This proves the stated form.
\end{proof}

Applying Theorem~\ref{thm:guillot} and
Lemma~\ref{lem:base-change} to \eqref{eq:ambient-derivation} gives a
subfield
\begin{equation}
 K=\C(h)\subset L
 \label{eq:base-field}
\end{equation}
with the following properties:
\begin{enumerate}
\item $K$ is relatively algebraically closed in $L$;
\item $Dh=b(h)$ for some $b\in\C(h)$;
\item $D$ is regular on the geometric generic fibre over $K$.
\end{enumerate}
The generic fibre is smooth, projective and geometrically integral.
No assumption on its genus is needed: the ramification argument will
determine it whenever $h$ and $t$ are algebraically independent.

Set $H=\iota(h)$. The function $H$ is nonconstant, because otherwise
injectivity would give $h\in\C$. Therefore
\begin{equation}
 H'=b(H),\qquad b\ne0.
 \label{eq:base-autonomous}
\end{equation}
More generally, there are no nonconstant rational first integrals of
$D$. If $DG=0$ with $G\in L$, then $\iota(G)$ is constant on $\C$;
injectivity again implies $G\in\C$.

\begin{remark}
\label{rem:time}
The identity $Dt=t$ persists under every birational change used here.
Multiplying $D$ by a nonconstant rational function would change this
identity, and would not preserve the time parametrization of the
trajectory. The argument below therefore uses a statement about the
vector field, not just its associated foliation.
\end{remark}

\subsection{The base trajectory}

The global meromorphic solution on the base can be classified without
any growth estimate.

\begin{lemma}
\label{lem:autonomous}
Let $H\in\MM(\C)$ be nonconstant and satisfy $H'=b(H)$ with
$b\in\C(s)$. A constant-coefficient fractional linear change of the
dependent variable transforms $H$ into either $z$ or $e^{\lambda z}$,
where $\lambda\in\C^*$.
\end{lemma}

\begin{proof}
View $b(s)\partial_s$ as a meromorphic vector field on $\PP^1$.
At a zero of this vector field, local uniqueness for a holomorphic
autonomous equation implies that a solution attaining that zero is
constant. A nonconstant solution also cannot attain a pole of the
vector field. Indeed, in a local coordinate vanishing at such a pole,
the derivative of a holomorphic map has no pole, whereas substitution
in the vector field would have one. This reasoning includes $s=\infty$
after replacing $s$ by $1/s$.

A nonzero rational vector field on $\PP^1$ has a zero or a pole.
Consequently $H$ omits at least one value. Picard's theorem says that
it omits at most two values; we use its meromorphic form as in
\cite{Hayman1964,Laine1993}.

If there is one omitted value, move it to infinity. The resulting
coefficient has neither zeros nor poles in $\C$, and is a nonzero
constant. Thus the transformed solution is $az+b_0$, $a\ne0$, and a
further affine change gives $z$.

If there are two omitted values, move them to $0$ and infinity. The
transformed solution is an entire, nowhere-zero function $G$, and its
equation has the form
\[
 G'=cG^j,\qquad c\ne0,\quad j\in\Z.
\]
For $j\ne1$, differentiation gives
\[
 G^{1-j}=(1-j)cz+c_0.
\]
The left-hand side is entire and nowhere zero, whereas the right-hand
side has a zero. Hence $j=1$. We obtain $G=c_1e^{cz}$, and division
by $c_1$ gives the required normalization.
\end{proof}

Applied to $H=\iota(h)$, the lemma changes only the coordinate within
$K=\C(h)$. It does not change $t$ or the independent variable $z$.
After this normalization, injectivity of $\iota$ yields respectively
\begin{equation}
 Dh=1\quad\text{or}\quad Dh=\lambda h.
 \label{eq:base-normal-forms}
\end{equation}

\section{Ramification of a rational eigenfunction}
\label{sec:ramification}

The next lemma is the local algebraic input. It applies to an arbitrary
smooth projective curve over a differential coefficient field.

\begin{lemma}
\label{lem:branch}
Let $\Gamma/K$ be a smooth projective geometrically integral curve and
let $D$ be regular on its geometric curve. Suppose that
$t\in K(\Gamma)\setminus K$ satisfies $Dt=t$. If
$a\in\overline K$ is a finite branch value of
$t:\Gamma_{\overline K}\to\PP^1_{\overline K}$, then
\begin{equation}
 Da=a.
 \label{eq:branch-equation}
\end{equation}
\end{lemma}

\begin{proof}
Choose a ramification point $p$ above $a$ and a uniformizer $\xi$ in
the discrete valuation ring $\OO_{\Gamma_{\overline K},p}$. If the
ramification index is $e\geq2$, write
\begin{equation}
 t-a=\xi^e v,
 \qquad v\in\OO_{\Gamma_{\overline K},p}^{\times}.
 \label{eq:local-ramification}
\end{equation}
Both $D\xi$ and $Dv$ lie in this ring. Therefore
\[
 D(t-a)=e\xi^{e-1}vD\xi+\xi^eDv\in(\xi).
\]
On the other hand, $D(t-a)=t-Da$. Reducing the latter identity modulo
the maximal ideal gives $a-Da=0$.
\end{proof}

\begin{remark}
It is not necessary to know that $D\xi\in(\xi)$. The factor
$\xi^{e-1}$ makes the first term vanish modulo $(\xi)$ precisely
because $e\geq2$. This is why preservation of the local ring suffices,
even though preservation of its maximal ideal need not hold.
\end{remark}

The local differential equation for the branch value becomes a global
algebraic obstruction once one uses the distinguished meromorphic
trajectory.

\begin{lemma}
\label{lem:no-finite-branch}
Suppose that $L$ has an injective differential homomorphism
$\iota:L\to\MM(\C)$ and that $h,t\in L$ are algebraically
independent, with
\[
 Dh\in\C(h),\qquad Dt=t,\qquad\iota(t)=e^z.
\]
Then the equation $Da=a$ has no nonzero solution algebraic over
$\C(h)$.
\end{lemma}

\begin{proof}
Suppose that such an $a$ exists. Set $H=\iota(h)$. Outside a finite set
in the $h$-sphere, the algebraic function $a$ has unramified
holomorphic branches, and a nonzero branch is nonvanishing away from
another finite set. Since the nonconstant meromorphic function $H$
omits at most two values, we may choose $z_0$ for which $H(z_0)$ avoids
all these exceptional values. The derivative of the resulting local
branch, composed with $H$, agrees with the uniquely extended derivation in
\eqref{eq:algebraic-derivative}. On a sufficiently small disc,
\[
 \frac{\mathrm d}{\mathrm dz}a(H(z))=a(H(z)).
\]
Consequently $a(H(z))=Ce^z$ there, for a nonzero constant $C$.

Let $P(X,Y)\in\C[X,Y]\setminus\{0\}$ be a polynomial relation for
$a$ over $\C(h)$. Then
\[
 P(H(z),Ce^z)=0
\]
on that disc, hence identically as a meromorphic identity. Injectivity
of $\iota$ gives $P(h,Ct)=0$ in $L$. Since $C\ne0$, this is a
nonzero polynomial relation between $h$ and $t$, a contradiction.
\end{proof}

Only a local branch of $a$ was used. No assertion that the algebraic
branch value extends single-valuedly over the time plane is required.
It is the polynomial identity, rather than the chosen branch, that
extends globally.

\begin{proposition}
\label{prop:two-point}
Let $K=\C(h)$ be relatively algebraically closed in a field $L$ of
transcendence degree one over $K$. Suppose that $L$ is the function
field of a smooth projective curve $\Gamma/K$, that $D$ is regular on
its geometric curve, and that $L$ admits an injective differential map
to $\MM(\C)$. If $h,t$ are algebraically independent,
$Dt=t$ and $\iota(t)=e^z$, then there exist
$u\in L$, $d\in\N$ and $c\in K^*$ such that
\begin{equation}
 L=K(u),\qquad t=c(h)u^d.
 \label{eq:kummer-form}
\end{equation}
\end{proposition}

\begin{proof}
Relative algebraic closedness and characteristic zero imply that the
generic curve is geometrically integral; compare
\cite[Chapter~III]{Stichtenoth2009}. The nonconstant function $t$
defines a finite separable map of degree $d$ on that curve. By
Lemmas~\ref{lem:branch} and~\ref{lem:no-finite-branch}, all geometric
branch values are contained in $\{0,\infty\}$.

Let $r_0$ and $r_\infty$ denote the respective numbers of geometric
points over $0$ and infinity. The contributions to the ramification
divisor over these two values are $d-r_0$ and $d-r_\infty$. Since
there is no other ramification, Riemann--Hurwitz gives
\begin{equation}
 2g(\Gamma)-2
 =-2d+(d-r_0)+(d-r_\infty)
 =-r_0-r_\infty.
 \label{eq:rh-two-point}
\end{equation}
The two point counts are positive, and the genus is nonnegative.
It follows that
\begin{equation}
 g(\Gamma)=0,\qquad r_0=r_\infty=1.
 \label{eq:genus-zero}
\end{equation}
We use Riemann--Hurwitz in its standard form for smooth projective
curves; see \cite[Chapter~IV, Section~2]{Hartshorne1977} or
\cite[Chapter~III]{Stichtenoth2009}.

Let $p_0,p_\infty$ be the unique geometric points above the two values.
They are fixed by the absolute Galois group of $K$, and therefore are
$K$-rational because $K$ is perfect. A genus-zero curve with a
$K$-rational point is isomorphic to $\PP^1_K$. More explicitly,
Riemann--Roch~\cite[Chapter~IV, Section~1]{Hartshorne1977} provides a function
with a simple pole at $p_\infty$; after subtracting its value at
$p_0$, it has a simple zero there. Choose this function as $u$.

The divisor of $t$ is $d(p_0-p_\infty)$, so $t/u^d$ has zero divisor.
On the geometrically integral projective curve its global regular
functions are $K$. Thus $t/u^d=c\in K^*$, proving
\eqref{eq:kummer-form} over $K$ itself.
\end{proof}

The descent of $p_0,p_\infty$ in this proof is useful: a description
of the cover only after extending $K$ would not yet give a meromorphic
function $U=\iota(u)$ on $\C$. Formula~\eqref{eq:kummer-form} is an
identity inside the original field $L$.

\section{Comparison with the original coefficient field}
\label{sec:comparison}

We now isolate the field-theoretic conclusion that completes the
algebraicity argument. Its assumptions are exactly the properties
obtained in Section~\ref{sec:guillot}.

\begin{proposition}
\label{prop:comparison}
Let $L=\C(t,w)$ with $t,w$ algebraically independent. Let $D$ be a
derivation with $Dt=t$, and suppose that an injective differential
homomorphism $\iota:L\to\MM(\C)$ satisfies $\iota(t)=e^z$.
Assume that there is a subfield $K=\C(h)$ relatively algebraically
closed in $L$ such that $D(K)\subset K$ and $D$ is regular on the
geometric generic curve of $L/K$. Then $Dw$ is a polynomial of degree
at most two in $w$, with coefficients in $\C(t)$.
\end{proposition}

\begin{proof}
We distinguish whether the two base functions are algebraically
dependent.

\smallskip\noindent
\emph{The dependent case.}
If $h$ is algebraic over $\C(t)$, Lemma~\ref{lem:relative} implies
$h\in\C(t)$. Since $h$ is nonconstant, $t$ is then algebraic over
$\C(h)$. Relative algebraic closedness of $K$ gives $t\in K$, hence
\begin{equation}
 K=\C(t).
 \label{eq:equal-bases}
\end{equation}
The smooth projective curve with field $L/K=\C(t)(w)$ is
$\PP^1_K$. Regularity and Lemma~\ref{lem:riccati} give the desired
conclusion in the original coordinate $w$.

\smallskip\noindent
\emph{The independent case.}
Suppose that $h,t$ are algebraically independent. By
Proposition~\ref{prop:two-point},
\begin{equation}
 L=\C(h,u),\qquad t=c(h)u^d.
 \label{eq:comparison-kummer}
\end{equation}
Put $H=\iota(h)$ and $U=\iota(u)$. Both are global meromorphic
functions. Lemma~\ref{lem:autonomous} permits a change of the base
coordinate so that $H=z$ or $H=e^{\lambda z}$. Such a change leaves
$K$ unchanged and merely rewrites the rational function $c$.

If $H=z$, substitution in \eqref{eq:comparison-kummer} gives
\begin{equation}
 e^z=c(z)U(z)^d.
 \label{eq:additive-divisor-identity}
\end{equation}
At every $a\in\C$, the left side has order zero. Therefore
\[
 \ord_a c+d\,\ord_a U=0.
\]
In particular, every finite zero and pole of $c$ has order divisible
by $d$. Factoring this rational function into linear factors yields
\begin{equation}
 c(h)=c_0r(h)^d,
 \qquad c_0\in\C^*,\quad r\in\C(h)^*.
 \label{eq:additive-divisibility}
\end{equation}
Set $v=r(h)u$. Then $L=\C(h,v)$ and $t=c_0v^d$.
Lemma~\ref{lem:relative} gives $d=1$. Consequently
\begin{equation}
 L=\C(t,h),\qquad Dh=1.
 \label{eq:additive-product}
\end{equation}
Use Corollary~\ref{cor:change-generator} over $\C(t)$, with $h$ as
the rational generator. It follows that $Dw$ is of Riccati type.

Suppose instead that $H=e^{\lambda z}$, $\lambda\ne0$. For every
$a\in\C^*$ there is a time $z_0$ with $H(z_0)=a$, and
$H'(z_0)=\lambda a\ne0$. Comparing orders in
\[
 e^z=c(H(z))U(z)^d
\]
at this time proves
\begin{equation}
 d\mid\ord_a c\qquad(a\in\C^*).
 \label{eq:multiplicative-divisibility}
\end{equation}
The two omitted base values are precisely the ones at which this
argument imposes no separate divisibility condition. Factorization in
$\C(h)$ now gives
\begin{equation}
 c(h)=c_0h^k r(h)^d,
 \qquad c_0\in\C^*,\quad k\in\Z,
 \quad r\in\C(h)^*.
 \label{eq:monomial-coefficient}
\end{equation}
With $v=r(h)u$ we obtain
\begin{equation}
 L=\C(h,v),\qquad t=c_0h^kv^d.
 \label{eq:monomial-map}
\end{equation}
Injectivity gives $Dh=\lambda h$. Differentiating the monomial
identity gives
\begin{equation}
 Dv=\mu v,\qquad \mu=\frac{1-k\lambda}{d}.
 \label{eq:diagonal-derivation}
\end{equation}

Let $g=\gcd(k,d)$, taken positive. If $g>1$, then the element
$h^{k/g}v^{d/g}\in L$ has $g$th power $t/c_0$. This contradicts
Lemma~\ref{lem:relative}. Thus $\gcd(k,d)=1$.
Choose integers $a,b$ with
\begin{equation}
 kb-da=1,
 \label{eq:bezout}
\end{equation}
and put $y=h^av^b$. The inverse change of variables is explicit:
\begin{equation}
 h=(t/c_0)^b y^{-d},\qquad
 v=(t/c_0)^{-a}y^k.
 \label{eq:inverse-monomial}
\end{equation}
Hence
\begin{equation}
 L=\C(t,y),\qquad Dy=(a\lambda+b\mu)y.
 \label{eq:original-base-product}
\end{equation}
Corollary~\ref{cor:change-generator}, again over the original base
$\C(t)$, proves the assertion for $Dw$.
\end{proof}

\begin{remark}[Two uses of relative algebraic closedness]
The condition on $\C(h)$ comes from connectedness of the new
fibration. The condition on $\C(t)$ comes independently from the
original presentation $L=\C(t)(w)$. The latter excludes roots of $t$
\emph{inside $L$}; it does not exclude the functions $e^{z/d}$ from
$\MM(\C)$. Confusing these two ambient fields would invalidate the
root argument.
\end{remark}

\begin{remark}
In \eqref{eq:monomial-map}, the condition $\gcd(k,d)=1$ states that
the exponent vector of $t$ is primitive in $\Z^2$. Equation
\eqref{eq:bezout} completes it to an integral basis. Thus the passage
to $y$ is a birational change on a two-dimensional algebraic torus,
not a finite covering. This is what allows the final generator change
to take place over $\C(t)$ without a further extension of that field.
\end{remark}

\begin{corollary}[Algebraicity of non-Riccati solutions]
\label{cor:algebraicity}
Under the hypotheses of Theorem~\ref{thm:main}, every meromorphic
solution $f$ is algebraic over $\C(e^z)$.
\end{corollary}

\begin{proof}
A transcendental solution would provide the injective differential map
\eqref{eq:evaluation}. Section~\ref{sec:guillot} supplies the subfield
$K$ and the regular generic curve required by
Proposition~\ref{prop:comparison}. That proposition forces
$Dw=R(t,w)$ to be of Riccati type, contrary to the hypothesis.
\end{proof}

Notice that Corollary~\ref{cor:algebraicity} does not assert that every
local solution of the equation is single-valued. It uses one specified
global meromorphic solution, and concludes that this solution is
algebraic over the coefficient field. Special algebraic trajectories
can meet the polar locus of the ambient vector field; they were never
subject to the Zariski-density argument.

\section{Algebraic descent and the exact deck period}
\label{sec:descent}

This section no longer uses a differential equation. The only
assumptions are algebraicity over $\C(e^z)$ and meromorphic
single-valuedness on $\C$.

\begin{lemma}[A finite integer period]
\label{lem:finite-period}
If $f\in\MM(\C)$ is algebraic over $\C(e^z)$, then
$f(z+2\pi i q)=f(z)$ for some positive integer $q$.
\end{lemma}

\begin{proof}
Choose $P(t,Y)\in\C[t,Y]\setminus\{0\}$ with
$P(e^z,f(z))=0$, and let $n=\deg_Y P$. Each function
\[
 f_j(z)=f(z+2\pi i j),\qquad j\in\Z,
\]
is a root of the same polynomial $P(e^z,Y)$ over the field
$\MM(\C)$. This polynomial is nonzero, since a nonzero polynomial
in $e^z$ cannot vanish identically. It has at most $n$ distinct roots.
Two of $f_0,\ldots,f_n$ are therefore identical, and their index
difference is a positive integer period as claimed.
\end{proof}

\begin{lemma}[Extension across the two punctures]
\label{lem:rational-descent}
Suppose that a meromorphic function $S$ on $\C^*$ is algebraic over
$\C(u)$. Then $S$ extends meromorphically to $\PP^1$ and is rational.
\end{lemma}

\begin{proof}
After clearing denominators, write
\[
 a_n(u)S(u)^n+a_{n-1}(u)S(u)^{n-1}+\cdots+a_0(u)=0,
 \qquad a_j\in\C[u],\quad a_n\ne0.
\]
On a sufficiently small punctured disc about zero, $a_n$ has no
zeros. The equation first implies that $S$ has no poles there: at
a pole the highest power, multiplied by a nonvanishing coefficient,
could not cancel. The elementary polynomial root bound then gives
\[
 |S(u)|\leq 1+\max_{j<n}\left|\frac{a_j(u)}{a_n(u)}\right|
 \leq C|u|^{-N}
\]
for suitable $C,N$. Thus zero is at most a pole of $S$. Applying the
same argument with $u=1/v$ treats infinity. Meromorphic functions on
the projective line are rational.
\end{proof}

\begin{proof}[Proof of Theorem~\ref{thm:primitive-intro}]
Let $q$ be the least positive integer supplied by
Lemma~\ref{lem:finite-period}. The exponential map
$z\mapsto u=e^{z/q}$ identifies $\C/(2\pi i q\Z)$ with $\C^*$;
see, for example, Forster~\cite[Section~5]{Forster1981} for the
covering-space description of the exponential map.
Consequently $f$ descends to a meromorphic function $S$ on $\C^*$.
If $P(e^z,f)=0$, then
\[
 P(u^q,S(u))=0.
\]
Lemma~\ref{lem:rational-descent} gives $S\in\C(u)$.

Put $E=\C(e^{z/q})$ and $F=\C(e^z)$. These fields form a cyclic
Galois extension of degree $q$; equivalently, this is the Kummer
extension $\C(u)/\C(u^q)$ in the standard function-field setting
\cite[Chapter~III]{Stichtenoth2009}. In the rational coordinate $u$, its
generator is
\[
 \sigma(u)=\zeta_q u,\qquad \zeta_q=e^{2\pi i/q}.
\]
The automorphism $\sigma^j$ fixes $f=S(u)$ exactly when
$f(z+2\pi i j)=f(z)$. Minimality of $q$ shows that the stabilizer
of $f$ is trivial. By Galois correspondence, $F(f)=E$. Hence
$m=[F(f):F]=q$, proving both assertions.
\end{proof}

\begin{corollary}
\label{cor:algebraic-field}
The relative algebraic closure of $\C(e^z)$ inside $\MM(\C)$ is
exactly $\Erad$.
\end{corollary}

\begin{proof}
Each $e^{z/q}$ is algebraic over $\C(e^z)$. The converse is
Theorem~\ref{thm:primitive-intro}.
\end{proof}

The minimal polynomial can also be read from the primitive
representation. If $f=S(e^{z/m})$, then
\begin{equation}
 \prod_{j=0}^{m-1}\bigl(Y-S(\zeta_m^j u)\bigr)
 \in\C(u^m)[Y]
 \label{eq:orbit-polynomial}
\end{equation}
is its monic minimal polynomial over $\C(e^z)$ after $u^m=e^z$.
All factors in the product are distinct. This describes the
algebraic degree without making any prior assumption that every
branch of an arbitrary plane equation is unramified over $\C^*$.

\begin{proof}[Proof of Theorem~\ref{thm:main}]
Corollary~\ref{cor:algebraicity} proves algebraicity over $\C(e^z)$.
Theorem~\ref{thm:primitive-intro} then gives the required expression.
Conversely, if $f$ is transcendental over $\C(e^z)$, the argument
leading to Corollary~\ref{cor:algebraicity} forces Riccati type.
\end{proof}

There is also a converse at the level of the class of functions,
where the equation is allowed to depend on the function.

\begin{proposition}
\label{prop:converse}
Every member of $\Erad$ satisfies at least one equation
$f'=R(e^z,f)$ with $R\in\C(t,w)$ not of Riccati type. Thus
$\Erad$ is exactly the union of the global meromorphic solution
sets of all such non-Riccati equations.
\end{proposition}

\begin{proof}
Let $P(t,w)\in\C(t)[w]$ be the monic minimal polynomial of $f$
over $\C(e^z)$, with $t$ replacing $e^z$. Separability gives
$P_w(e^z,f)\not\equiv0$. Differentiating the algebraic relation
gives
\[
 f'=-\frac{e^z P_t(e^z,f)}{P_w(e^z,f)}.
\]
Here $P_t$ denotes the ordinary derivative of the rational
coefficients with respect to $t$. Therefore the rational function
\begin{equation}
 R(t,w)=-\frac{tP_t(t,w)}{P_w(t,w)}+P(t,w)w^3
 \label{eq:converse-equation}
\end{equation}
has the required solution. If $n=\deg_w P$, monicity implies
$\deg_w P_t\leq n-1$ and $\deg_w P_w=n-1$. The first term in
\eqref{eq:converse-equation} is bounded at $w=\infty$, while the
second has leading term $w^{n+3}$. Thus the reduced $R$ has a pole
of order $n+3\geq4$ at $w=\infty$ and cannot be of Riccati type.
\end{proof}

\begin{corollary}
\label{cor:rescaled}
Let $\alpha\in\C^*$ and $\beta\in\C$. If
$f'=R(e^{\alpha z+\beta},f)$, where $R\in\C(t,w)$ is not of
Riccati type, then $f(z)=S(e^{\alpha z/q})$ for some rational $S$
and some $q\in\N$.
\end{corollary}

\begin{proof}
Set $\zeta=\alpha z$ and $F(\zeta)=f(\zeta/\alpha)$.
Then $F'=\alpha^{-1}R(e^\beta e^\zeta,F)$. The new rational
function remains non-Riccati, so Theorem~\ref{thm:main} applies.
\end{proof}

\subsection{The normalized algebraic curve}

The primitive representation determines more than the function itself.
It also identifies the finite curve over the coefficient line, including
its ramification and its automorphism group.

\begin{corollary}[Cyclic normalization]
\label{cor:cyclic-normalization}
Let $f$ be algebraic over $\C(e^z)$ of degree $m$. Let $C_f$ be
the smooth projective curve with function field $\C(e^z,f)$,
viewed as a finite extension of $\C(t)$ by $t=e^z$. There is a
coordinate $u$ on $C_f\simeq\PP^1$ for which
\[
 t=u^m.
\]
The resulting cover is cyclic of degree $m$, is unramified over
$\C^*$, and has ramification divisor
$(m-1)([0]+[\infty])$. In particular, every algebraic conjugate of
$f$ over $\C(e^z)$ is again a global meromorphic function on $\C$.
\end{corollary}

\begin{proof}
The field equality in Theorem~\ref{thm:primitive-intro} identifies
the coordinate $u$ with $e^{z/m}$ and gives $t=u^m$. Smooth
projective curves are determined by their function fields. The
ramification indices at zero and infinity are both $m$, and the
derivative $m u^{m-1}$ is nonzero on $\C^*$. Multiplication of $u$
by an $m$th root of unity gives the cyclic Galois group. The
conjugates of $f=S(u)$ are the functions
$S(\zeta_m^j e^{z/m})=f(z+2\pi i j)$, which are meromorphic.
\end{proof}

The use of normalization in this corollary is essential when a
solution is presented by an implicit polynomial. A plane equation
$P(t,w)=0$ can have singular points over $t_0\ne0$. The vanishing
of $P_w$ at such a point does not imply ramification of its
normalization over the $t$-line. Distinct analytic branches can
intersect in the plane while remaining distinct and unramified on
the normalized curve.

This distinction also clarifies the quotient $-tP_t/P_w$ in
\eqref{eq:converse-equation}. Its apparent singularities along a
particular trajectory may cancel after restriction to the
normalization. An ambient pole of the rational vector field and
a pole of the time-parametrized solution are different objects.
Neither the algebraicity proof nor the descent argument assumes
that the ambient plane equation is nonsingular.

The corollary can alternatively be interpreted in terms of the
fundamental group of $\C^*$. Its connected finite covers correspond
to finite-index subgroups of $\Z$, and are represented by
$u\mapsto u^m$~\cite[Sections~4--5]{Forster1981}. In the present proof
this covering description is a
consequence of single-valued algebraic descent. It was not assumed
before the integer period had been proved.

\section{Growth and value fibres}
\label{sec:growth}

We use the standard characteristic $T(r,f)=m(r,f)+N(r,f)$, with
$N(r,f)$ counting poles with multiplicities and with the usual
convention at the origin. General background on these quantities can
be found in~\cite{Hayman1964,Laine1993}. In the present situation the
growth formula follows from a short direct calculation.

\begin{proposition}
\label{prop:growth}
Let $q\in\N$ and let $S\in\C(u)$ be nonconstant of degree $n$.
Then
\begin{equation}
 T(r,S(e^{z/q}))=\frac{n}{\pi q}r+O(1)
 \qquad(r\to\infty).
 \label{eq:exact-growth}
\end{equation}
\end{proposition}

\begin{proof}
Write $S=P/Q$ with relatively prime polynomials, and let
$n=\max(\deg P,\deg Q)$. There are positive constants $c,C$ such that
\begin{equation}
 c\max(1,|u|)^n
 \leq\max(|P(u)|,|Q(u)|)
 \leq C\max(1,|u|)^n
 \label{eq:polynomial-norm}
\end{equation}
for every $u\in\C$. At infinity this follows from the leading
terms, and on a compact set it follows from the absence of common
zeros.

Set $g=e^{z/q}$. The entire functions $P(g)$ and $Q(g)$ have no
common zeros. Jensen's formula for $Q(g)$, together with the
definition of $m(r,P(g)/Q(g))$, gives
\[
 T(r,S(g))=
 \frac1{2\pi}\int_0^{2\pi}
 \log\max\{|P(g(re^{i\theta}))|,|Q(g(re^{i\theta}))|\}
 \dd\theta+O(1).
\]
If $Q(g)$ vanishes at the origin, its first nonzero Taylor
coefficient is used in Jensen's formula; this only changes the
constant. By \eqref{eq:polynomial-norm}, the last expression is
\[
 \frac{n}{2\pi}\int_0^{2\pi}
 \max\{0,(r/q)\cos\theta\}\dd\theta+O(1)
 =\frac{n}{\pi q}r+O(1).
\]
\end{proof}

\begin{corollary}
\label{cor:order}
Every nonconstant solution in Theorem~\ref{thm:main} has order exactly
one and finite positive type. Moreover,
\[
 \pi\lim_{r\to\infty}\frac{T(r,f)}r\in\Q_{>0}.
\]
\end{corollary}

The rational number in this formula is independent of whether the
chosen representation is primitive. Replacing $u$ by an integer
power increases the degree of the rational function and the
denominator of the exponent by the same factor.

\begin{corollary}
\label{cor:entire}
Every entire solution of a non-Riccati equation
\eqref{eq:main-equation} is a Laurent polynomial in a finite root
of $e^z$. Thus it can be written as
\begin{equation}
 f(z)=\sum_{j=-a}^{b}c_j e^{jz/q},
 \qquad a,b\geq0,\quad q\geq1,
 \label{eq:laurent-solution}
\end{equation}
with constant coefficients $c_j$.
\end{corollary}

\begin{proof}
Write $f=S(e^{z/q})$. Since the exponential map covers $\C^*$,
the rational function $S$ cannot have a pole at a nonzero finite
point. A rational function with poles confined to $0$ and infinity
is a Laurent polynomial.
\end{proof}

For a nonconstant meromorphic function $f$ and $a\in\PP^1$, define
its \emph{exponential fibre set} by
\begin{equation}
 E_a(f)=\{e^z:z\in\C,\ f(z)=a\}.
 \label{eq:exponential-fibre}
\end{equation}
For $a=\infty$, the preimage means the set of poles.

\begin{proposition}
\label{prop:fibres}
If $f=S(e^{z/q})$ with $S$ nonconstant and rational, then every
$E_a(f)$ is finite. More precisely,
\begin{equation}
 E_a(f)=\{\alpha^q:\alpha\in\C^*,\ S(\alpha)=a\},
 \label{eq:fibre-formula}
\end{equation}
and
\begin{equation}
 f^{-1}(a)=
 \bigcup_{\substack{\alpha\in\C^*\\S(\alpha)=a}}
 \bigl(q\log\alpha+2\pi iq\Z\bigr).
 \label{eq:lattice-fibres}
\end{equation}
Any logarithm may be used in each coset. The multiplicity is constant
on each coset and equals the local multiplicity of $S$ at $\alpha$.
\end{proposition}

\begin{proof}
The equation $S(e^{z/q})=a$ is equivalent to
$e^{z/q}=\alpha$ for one of the finitely many nonzero finite
preimages $\alpha$ of $a$ under $S$. Solving this exponential
equation gives \eqref{eq:lattice-fibres}, and exponentiating $z$
gives \eqref{eq:fibre-formula}. Since $e^{z/q}$ has nonzero
derivative at every finite time, local multiplicities are preserved.
\end{proof}

In particular, $|E_a(f)|\leq\deg S$. This holds for all values,
although the number of distinct exponential images can be smaller
than the number of preimages of $a$ under $S$: different $\alpha$
may have the same $q$th power.

\subsection{What three finite exponential fibres imply}

The preceding conclusion is stronger than a restriction at three
values. It is useful to record exactly what can be deduced from that
weaker condition without a differential equation.

\begin{proposition}
\label{prop:three-fibres}
Suppose that $f$ is nonconstant and meromorphic in $\C$. If
$E_{a_j}(f)$ is finite for three distinct values
$a_1,a_2,a_3\in\PP^1$, then $T(r,f)=O(r)$.
\end{proposition}

\begin{proof}
For each $j$, the support of the $a_j$-points is contained in a
finite union of progressions $b+2\pi i\Z$. The corresponding
truncated, unintegrated counting function is $O(r)$, and hence
\[
 \overline N(r,a_j;f)=O(r).
\]
The three-value second main theorem gives
\begin{equation}
 T(r,f)\leq Cr+
 C\bigl(\log^+T(r,f)+\log r\bigr)
 \label{eq:smt-bound}
\end{equation}
outside a set of finite linear measure. We use the standard error
term in~\cite{Hayman1964,Laine1993}. For large positive $x$,
$C\log^+x\leq x/2+C_0$. Absorbing this term in
\eqref{eq:smt-bound} proves $T(r,f)=O(r)$ outside the exceptional
set. Its tail has measure less than one; for every sufficiently
large $r$ there is a nonexceptional $s\in[r,r+1]$. Monotonicity of
the characteristic extends the estimate to every large $r$.
\end{proof}

This proposition supplies a growth bound but does not supply a period
or algebraicity. The following example makes the distinction explicit.

\begin{example}[Three fibres without exponential descent]
\label{ex:three-fibres}
The power series
\begin{equation}
 F(z)=\sum_{n=0}^{\infty}\frac{(i\pi z/2)^n}{(2n)!}
     =\cos\sqrt{-i\pi z/2}
 \label{eq:cosine-root}
\end{equation}
defines a single-valued entire function of order $1/2$. It satisfies
\begin{equation}
 E_1(F)=E_{-1}(F)=\{1\},\qquad E_0(F)=\{i\},
 \label{eq:three-special-fibres}
\end{equation}
but $F\notin\Erad$.
\end{example}

\begin{proof}
The square-root notation abbreviates the entire power series; the
cosine is even, so the branch choice has no effect. The estimate
$|F(z)|\leq\exp(C\sqrt{|z|})$ gives order at most $1/2$. On the
ray $z=-ir$, $F(-ir)=\cosh\sqrt{\pi r/2}$, giving the reverse
bound.

If the cosine argument is $s$, then $z=2is^2/\pi$. Its values
$\pm1$ occur for integer multiples of $\pi$, and its zeros for
$s=\pi(n+1/2)$. Exponentiating the corresponding $z$ gives
\eqref{eq:three-special-fibres}. The distinct zeros have counting
function $O(\sqrt r)$. A nonzero period would repeat any one zero
along an arithmetic progression and give at least a positive
multiple of $r$ zeros in large discs, a contradiction. Thus $F$
has no nonzero period and cannot belong to $\Erad$.
\end{proof}

The example shows why the identity $Dt=t$ and the differential
equation remain essential in the proof of
Theorem~\ref{thm:main}. Finiteness of three exponential fibre sets,
even together with finite order, is not a substitute for that input.

\subsection{Sectorial growth of entire solutions}

For an entire solution in \eqref{eq:laurent-solution}, the two
extreme exponents describe its directional growth. We state this
elementary consequence to distinguish its geometric growth data
from the characteristic slope of Proposition~\ref{prop:growth}.

\begin{proposition}[Indicator of an entire solution]
\label{prop:indicator}
Suppose that $f$ is nonconstant and has the representation
\[
 f(z)=\sum_{j=\ell}^{J}c_j e^{jz/q},
 \qquad c_\ell c_J\ne0,
\]
where missing intermediate coefficients are allowed. Its order-one
indicator is
\begin{equation}
 h_f(\theta):=\limsup_{r\to\infty}
       \frac{\log|f(re^{i\theta})|}{r}
 =\max_{c_j\ne0}\frac{j}{q}\cos\theta.
 \label{eq:indicator}
\end{equation}
On a closed angular sector on which $\cos\theta$ is positive,
the $J$th term gives a uniform asymptotic equivalent. On a closed
sector on which $\cos\theta$ is negative, the $\ell$th term does.
\end{proposition}

\begin{proof}
If $\cos\theta\geq\varepsilon>0$, factor out
$c_J e^{Jz/q}$. Every remaining exponential has modulus bounded
by $C e^{-\varepsilon r/q}$, so
\[
 f(re^{i\theta})=c_J e^{Jre^{i\theta}/q}
 \bigl(1+O(e^{-\varepsilon r/q})\bigr)
\]
uniformly on the sector. Factoring the lowest exponent proves
the analogous statement for $\cos\theta\leq-\varepsilon$.
If there is only one term, these statements hold as exact
identities.

On either imaginary ray, the function is bounded, so the indicator
is at most zero. Its restriction to that line is periodic and is
not identically zero. A nonzero value therefore repeats at
arbitrarily large radii, making the indicator at least zero.
This proves \eqref{eq:indicator} also on the two boundary rays.
\end{proof}

In this representation the rational degree of the Laurent
polynomial is
\[
 n=\max(J,0)-\min(\ell,0).
\]
Consequently the characteristic slope depends on the interval
spanned by the frequencies together with zero. The indicator
itself can be negative in a half-plane when all exponents have
the same strict sign. These two descriptions are consistent,
since the characteristic uses $\log^+|f|$.

\section{Deficiencies and ramification}
\label{sec:deficiencies}

The rational representation also determines the leading term of
every counting function. In this section let
\[
 f(z)=S(e^{z/q}),\qquad n=\deg S\geq1.
\]
For $a\in\PP^1$, write $N(r,a;f)$ for the integrated counting
function of the $a$-points, with multiplicity, and
$\overline N(r,a;f)$ for its truncated version. Thus
$N(r,\infty;f)=N(r,f)$.

For $p\in\PP^1$, let $e_p(S)$ be the local mapping degree of
$S$ at $p$. Define
\begin{equation}
 \epsilon_p(a)=
 \begin{cases}
 e_p(S),&S(p)=a,\\
 0,&S(p)\ne a,
 \end{cases}
 \qquad p\in\{0,\infty\},
 \label{eq:endpoint-multiplicities}
\end{equation}
and let
\[
 \nu_a=\#\{\alpha\in\C^*:S(\alpha)=a\}.
\]
The multiplicities $\epsilon_0(a)$ and $\epsilon_\infty(a)$
record exactly the preimages lost when the rational curve is
restricted to the exponential image $\C^*$.

\begin{theorem}[Counting functions and defects]
\label{thm:defects}
For each fixed $a\in\PP^1$,
\begin{align}
 N(r,a;f)
 &=\frac{n-\epsilon_0(a)-\epsilon_\infty(a)}{\pi q}
       r+O(\log r),\label{eq:counting-asymptotic}\\
 \overline N(r,a;f)
 &=\frac{\nu_a}{\pi q}r+O(\log r).
 \label{eq:truncated-asymptotic}
\end{align}
Accordingly, the Nevanlinna deficiency and the truncated deficiency,
defined respectively by
\[
 \delta(a,f)=1-\limsup_{r\to\infty}\frac{N(r,a;f)}{T(r,f)},
 \qquad
 \Theta(a,f)=1-\limsup_{r\to\infty}
                 \frac{\overline N(r,a;f)}{T(r,f)},
\]
are given by
\begin{equation}
 \delta(a,f)=\frac{\epsilon_0(a)+\epsilon_\infty(a)}{n},
 \qquad \Theta(a,f)=1-\frac{\nu_a}{n}.
 \label{eq:defect-formulas}
\end{equation}
In particular, the only possible deficient values are $S(0)$ and
$S(\infty)$, and
\begin{equation}
 \sum_{a\in\PP^1}\delta(a,f)
   =\frac{e_0(S)+e_\infty(S)}{n},
 \qquad
 \sum_{a\in\PP^1}\Theta(a,f)=2.
 \label{eq:defect-sums}
\end{equation}
The sums have finite support. If $S(0)=S(\infty)$, their
contributions to $\delta$ occur at the same value.
\end{theorem}

\begin{proof}
A progression $b+2\pi iq\Z$ contains
\begin{equation}
 \frac{r}{\pi q}+O(1)
 \label{eq:progression-count}
\end{equation}
points in the disc of radius $r$. Indeed, the admissible integers
form an interval of length
$\sqrt{r^2-(\operatorname{Re}b)^2}/(\pi q)$, up to an endpoint
error bounded independently of $r$. This length is
$r/(\pi q)+O(1/r)$ for large $r$.

The progressions in Proposition~\ref{prop:fibres} are disjoint
for distinct $\alpha$. Their total multiplicity is
\[
 d_a=\sum_{\substack{\alpha\in\C^*\\S(\alpha)=a}}
              e_\alpha(S)
     =n-\epsilon_0(a)-\epsilon_\infty(a).
\]
Thus their unintegrated counting functions, counted with and
without multiplicity, have slopes $d_a/(\pi q)$ and
$\nu_a/(\pi q)$, respectively. Integrating
\eqref{eq:progression-count} from a fixed radius to $r$ gives
\eqref{eq:counting-asymptotic} and
\eqref{eq:truncated-asymptotic}; the bounded counting error
integrates to $O(\log r)$. The standard convention for a point
at the origin gives the same conclusion.

Divide these formulas by Proposition~\ref{prop:growth}. This
proves \eqref{eq:defect-formulas}, including the existence of the
ratios as actual limits. Summing the first formula gives the
first equality in \eqref{eq:defect-sums}.

For the second equality, the degree formula for each fibre of
$S$ gives
\begin{equation}
 n-\nu_a=
 \sum_{\substack{\alpha\in\C^*\\S(\alpha)=a}}
         (e_\alpha(S)-1)
 +\epsilon_0(a)+\epsilon_\infty(a).
 \label{eq:truncated-defect-local}
\end{equation}
Only critical values and the two endpoint images can contribute.
Riemann--Hurwitz for $S:\PP^1\to\PP^1$ reads
\[
 \sum_{p\in\PP^1}(e_p(S)-1)=2n-2.
\]
Summing \eqref{eq:truncated-defect-local} over all values and
removing the ramification contributions of zero and infinity
therefore yields
\begin{align*}
 \sum_a(n-\nu_a)
 &=2n-2-(e_0(S)-1)-(e_\infty(S)-1)
                  +e_0(S)+e_\infty(S)\\
 &=2n.
\end{align*}
Division by $n$ proves the second equality.
\end{proof}

The two kinds of deficiency have different geometric sources.
The ordinary deficiencies arise from the two missing points of
the exponential cover. Truncation also records the ramification
of $S$ within $\C^*$. The last equality in
\eqref{eq:defect-sums} expresses the total ramification of
the rational map together with its two missing endpoints. It
realizes equality in the usual truncated defect bound from the
second main theorem~\cite{Hayman1964,Laine1993}.

\begin{corollary}
\label{cor:deficient-values}
Every nonconstant meromorphic solution of a non-Riccati equation
\eqref{eq:main-equation} has at least one and at most two
deficient values. Their deficiencies are positive rational
numbers. A value $a$ is omitted exactly when
\[
 \epsilon_0(a)+\epsilon_\infty(a)=n.
\]
For an entire nonconstant solution, infinity has deficiency one.
\end{corollary}

\begin{proof}
At least one endpoint image exists and each endpoint has
positive local degree. The formulas in Theorem~\ref{thm:defects}
give the assertions about the deficiencies. There are no
$a$-points precisely when $d_a=0$, which is the displayed
condition. If $f$ is entire, all poles of $S$ lie at the two
endpoints, with total multiplicity $n$.
\end{proof}

For example, take $S(u)=u^2+u$ and $q=1$. The endpoint images
are $0$ and infinity; the additional finite critical value is
$-1/4$. The formulas give
\[
 \begin{array}{c|ccc}
 a&0&-1/4&\infty\\ \hline
 \delta(a,f)&1/2&0&1\\
 \Theta(a,f)&1/2&1/2&1
 \end{array}
\]
and both quantities vanish at all other values. The function
$f=e^{2z}+e^z$ is an entire solution of the non-Riccati equation
\[
 f'=2e^{2z}+e^z+(f-e^{2z}-e^z)^3.
\]
Thus partial deficiency and the additional effect of truncation
already occur within the exact solution class of the main theorem.

\section{Examples and sharpness}
\label{sec:examples}

\subsection{Every denominator occurs}

For $q\geq2$, the function
\begin{equation}
 f_q(z)=e^{-z/q}
 \label{eq:all-denominators}
\end{equation}
satisfies
\begin{equation}
 f_q'=-\frac1q e^z f_q^{q+1}.
 \label{eq:all-denominators-equation}
\end{equation}
The rational function $R(t,w)=-tw^{q+1}/q$ is non-Riccati, and
the algebraic degree of $f_q$ over $\C(e^z)$ is exactly $q$.
Indeed, $\C(e^z,f_q)=\C(e^{z/q})$. Thus no fixed integer
denominator can suffice for all equations in the theorem.
The case $q=2$ is \eqref{eq:historical-example}.

The possible positive characteristic slopes also have no positive
lower bound independent of the equation. In fact, every positive
rational number occurs as $\pi\lim T(r,f)/r$. If $p,q$ are positive
integers, then $f=e^{pz/q}$ satisfies the non-Riccati equation
\begin{equation}
 f'=\frac pq e^{-2pz} f^{2q+1},
 \label{eq:rational-types}
\end{equation}
whose coefficient is rational in $e^z$. Its characteristic is
$pr/(\pi q)$.

\subsection{A non-Riccati equation with a denominator in the dependent variable}

Let $C_q$ be the polynomial determined by
\[
 C_0(w)=2,\qquad C_1(w)=w,\qquad
 C_{q+1}(w)=wC_q(w)-C_{q-1}(w).
\]
An induction gives
\begin{equation}
 C_q(u+u^{-1})=u^q+u^{-q}.
 \label{eq:chebyshev-identity}
\end{equation}
For $q\geq2$, take
\begin{equation}
 f(z)=e^{z/q}+e^{-z/q}.
 \label{eq:two-sided-example}
\end{equation}
Differentiating \eqref{eq:chebyshev-identity} yields
\begin{equation}
 f'=\frac{e^z-e^{-z}}{C_q'(f)}.
 \label{eq:chebyshev-equation}
\end{equation}
The corresponding $R(t,w)=(t-t^{-1})/C_q'(w)$ has a nonconstant
denominator in $w$ and is not of Riccati type. There is no common
factor to cancel over $\C(t)[w]$, because the numerator is a
nonzero element of the coefficient field. The solution itself is
entire; apparent singularities in the substituted quotient are
removable by the differentiated identity.

For $q=3$, the example reads
\[
 f^3-3f=e^z+e^{-z},\qquad
 f'=\frac{e^z-e^{-z}}{3(f^2-1)}.
\]
The integer deck period is exactly $q$, since the only
$q$th root of unity $\zeta$ satisfying
$\zeta u+\zeta^{-1}u^{-1}=u+u^{-1}$ is $\zeta=1$.
Theorem~\ref{thm:primitive-intro} gives algebraic degree $q$,
whereas Proposition~\ref{prop:growth} gives characteristic slope
$2/(\pi q)$.

\subsection{Why Riccati equations must be excluded}

The example recorded in~\cite[Question~6.2]{Gundersen2017} is
\begin{equation}
 f(z)=\frac{e^z}{z(e^z-1)}.
 \label{eq:riccati-counterexample}
\end{equation}
A direct logarithmic differentiation gives
\begin{equation}
 f'=\frac{1}{1-e^z}f+\frac{1-e^z}{e^z}f^2.
 \label{eq:counterexample-riccati}
\end{equation}
For every nonzero integer $q$,
\[
 f(z+2\pi iq)
 =\frac{e^z}{(z+2\pi iq)(e^z-1)}\not\equiv f(z).
\]
Thus $f$ cannot belong to $\Erad$. It has finite order; finite
order alone does not extend the main theorem to the Riccati class.

The linear equation $f'=e^zf$ gives the separate example
$f=e^{e^z}$, of infinite order. It is $2\pi i$-periodic, but the
descended function $u\mapsto e^u$ has an essential singularity at
infinity. This illustrates the independent role of algebraicity
in Lemma~\ref{lem:rational-descent}.

\subsection{A field in which the root obstruction fails}

Take an abstract field $L=\C(h,u)$ with
\[
 Dh=1,\qquad Du=u/d,\qquad t=u^d,
 \qquad d\geq2.
\]
The substitution $h=z$, $u=e^{z/d}$ gives an injective differential
homomorphism into $\MM(\C)$, and $Dt=t$. There is a regular
fibration over $\C(h)$ and its two-point description has degree
$d$. However, $\C(t)$ is not relatively algebraically closed in
$L$, because it has the proper algebraic extension $\C(u)$ inside
$L$. In particular, $L$ cannot have a presentation
$L=\C(t)(w)$ with $w$ transcendental over $\C(t)$.

This example identifies the exact obstruction excluded in
Proposition~\ref{prop:comparison}. A rational surface alone does
not make every rational function on it a rational-fibration
coordinate. The original presentation of the equation supplies
that additional information.

\section{The additive coordinate and Malmquist's theorem}
\label{sec:additive}

The same ramification argument also explains the difference between
rational and exponential coefficient fields. We give the comparison
as a recovery of the classical theorem, with its original attribution
to Malmquist~\cite{Malmquist1913}. It is not a separate claim of a new
Malmquist theorem.

\begin{lemma}
\label{lem:additive-branch}
In Lemma~\ref{lem:branch}, replace $Dt=t$ by $Dt=1$. Then every
finite branch value $a$ satisfies $Da=1$.
\end{lemma}

\begin{proof}
Use the same expression $t-a=\xi^e v$. Its derivative is in the
maximal ideal, while $D(t-a)=1-Da$ belongs to the coefficient field.
Reduction modulo the maximal ideal gives the assertion.
\end{proof}

Suppose in addition that a distinguished differential embedding
satisfies $\iota(t)=z$ and that $h,t$ are algebraically independent.
If a finite branch value $a$ existed, a local branch along the
trajectory would satisfy
\[
 a(H(z))=z+C.
\]
A polynomial relation for $a$ over $\C(h)$ would give a polynomial
relation between $h$ and $t+C$, which is impossible. Thus the map
$t$ on the geometric generic fibre can have at most one branch
value, namely infinity.

\begin{lemma}
\label{lem:one-point}
A finite morphism from a smooth projective connected curve to
$\PP^1$ over an algebraically closed field of characteristic zero
has degree one if it is unramified outside at most one point.
\end{lemma}

\begin{proof}
If the map has a branch value, choose the sole possible one; otherwise
choose an arbitrary point of the target. Let the degree be $d$ and let
$r\geq1$ be the number of points above the chosen value. The total
ramification is $d-r$ (zero in the unramified case), so
Riemann--Hurwitz gives
\[
 2g-2=-2d+(d-r)=-d-r.
\]
Since the left side is at least $-2$, it follows that $d+r\leq2$.
Thus $d=r=1$ and $g=0$.
\end{proof}

\begin{theorem}
\label{thm:additive}
If $R\in\C(t,w)$ is not of Riccati type, then every global
meromorphic solution of $f'=R(z,f)$ is rational.
\end{theorem}

\begin{proof}
Suppose first that $f$ is transcendental over $\C(z)$. The
embedding $t\mapsto z$, $w\mapsto f$ and the derivation
\[
 D=\partial_t+R(t,w)\partial_w
\]
produce a dense trajectory on $\PP^1\times\PP^1$. The argument
of Section~\ref{sec:guillot} gives a preserved fibration with base
$K=\C(h)$ and a regular geometric generic curve.

If $h,t$ are dependent, the two relative algebraic closedness
arguments give $K=\C(t)$, and Lemma~\ref{lem:riccati} applies
directly. If they are independent, additive branch compatibility
and the preceding exclusion leave at most one branch value.
Lemma~\ref{lem:one-point} gives degree one, hence $L=K(t)$.
By Lemma~\ref{lem:autonomous}, a constant fractional linear change
of $h$ gives $Dh=1$ or $Dh=\lambda h$. In either case,
$L=\C(t)(h)$ is a rational differential field with a regular
derivation in $h$. Corollary~\ref{cor:change-generator} forces
Riccati type in $w$, a contradiction.

Thus $f$ is algebraic over $\C(z)$. After clearing denominators, choose
a polynomial relation for $f$ with coefficients in $\C[z]$. For all
sufficiently large $|z|$, its leading coefficient is nonzero and the
polynomial root bound used in Lemma~\ref{lem:rational-descent} gives
$|f(z)|\leq C|z|^N$ for some $C,N>0$. Hence infinity is at most a pole.
Since $f$ is already meromorphic at every finite point, it extends
meromorphically to $\PP^1$ and is therefore rational.
\end{proof}

For the additive flow there is one omitted value, infinity, and the
cover is forced to have degree one. For the multiplicative flow
there are two omitted values, and the intermediate cover can have
the form $u\mapsto u^d$. This accounts for the finite root
extensions in the exponential-coefficient theorem. The subsequent
descent is controlled by the deck group of $\C^*$, rather than by
any growth restriction on the original solution.

\enlargethispage{3\baselineskip}

\section*{Acknowledgement}
This work was supported by Guangdong Basic and Applied Basic Research Foundation
(No.~2022A1515110967 and No.~2023A1515011809). The authors acknowledges the use of AI-assisted tools in manuscript preparation and assumes full responsibility for all mathematical statements and proofs.

\end{document}